\documentclass[11pt]{amsart}
\newcounter{thmlistcnt}
\newenvironment{thmlist}%
	{\setcounter{thmlistcnt}{0}%
	\begin{list}{\emph{(\roman{thmlistcnt})}}{%
		\usecounter{thmlistcnt}%
		\setlength{\topsep}{0pt}%
		\setlength{\leftmargin}{0pt}%
		\setlength{\itemsep}{0pt}%
		\setlength{\labelwidth}{17pt}
		\setlength{\itemindent}{30pt}}%
	}%
	{\end{list}}%
	
\makeatletter
\@namedef{subjclassname@2020}{\textup{2020} Mathematics Subject Classification}
\makeatother	

\newcommand{\mfrac}[2]{{\textstyle\frac{#1}{#2}}}

\newcommand{\Sym}{\mathrm{Sym}}
\usepackage{amsmath, amssymb}
\newtheorem{theorem}{Theorem}[section]
\newtheorem{lemma}[theorem]{Lemma}
\newtheorem{corollary}[theorem]{Corollary}

\theoremstyle{remark}
\newtheorem{remark}[theorem]{Remark}

\renewcommand{\P}{\mathbb{P}}
\newcommand{\deltaP}{\varepsilon}
\renewcommand{\theta}{\vartheta} 
\newcommand{\ts}{s} 
\newcommand{\zero}{0} 

\theoremstyle{definition}
\newtheorem{example}[theorem]{Example}

\title[Double coset lumping]{A necessary and sufficient condition for double coset lumping of Markov chains on groups with an application to the random to top shuffle}
\author{John R.~Britnell and Mark Wildon}
\address{School of Mathematics, Statistics and Physics, Newcastle University, Newcastle Upon Tyne, NE1 7RU}
\email{John.Britnell1@newcastle.ac.uk}

\address{Heilbronn Institute for Mathematical Research, School of Mathematics, University of Bristol, Woodland Road, Bristol BS8 1UG, UK}
\email{mark.wildon@bristol.ac.uk}


\newcommand{\N}{\mathbb{N}}
\begin{document}

\begin{abstract}
Let $Q$ be a probability measure on a finite group $G$, and let~$H$ be a subgroup of $G$. We show that a necessary and sufficient condition for the random walk driven by $Q$ on $G$ to induce a Markov chain on the double coset space $H\backslash G/H$ is that $Q(gH)$ is constant as~$g$ ranges over any double coset of $H$ in $G$. We obtain this result as a corollary of a more general theorem on the double cosets $H \backslash G / K$ for~$K$ an arbitrary subgroup of $G$. As an application we study a variation on the $r$-top to random shuffle which we show induces an irreducible, recurrent, reversible and ergodic Markov chain on the  double cosets of $\Sym_r \times \Sym_{n-r}$ 
in $\Sym_n$. The transition matrix of the induced walk has remarkable spectral properties: we find its invariant distribution and its eigenvalues and hence determine its rate of convergence.
\end{abstract}

\maketitle

\section{Introduction}

\thispagestyle{empty}
Diaconis, Ram and Simper \cite{DRS} have shown that if $H$ and $K$ are subgroups of a finite group $G$, and $Q$ is a probability measure on $G$ which is constant on the conjugacy classes of $G$, then the random walk on $G$ driven by~$Q$ induces a time-homogeneous
Markov chain on the double cosets $H\backslash G/K$ under the quotient map 
sending $g \in G$ to its double coset $HgK$. This is notable as typically a quotient of a Markov chain does not have the Markov property.
In this paper we state and prove Theorem~\ref{thm:HxK}, 
giving  a necessary and sufficient condition for the random walk 
driven by an arbitrary probability measure on $G$ to induce a time-homogeneous Markov chain on the 
double cosets 
$H\backslash G/K$. (The chain may start at an arbitrary initial distribution;
this includes a deterministic state.)
The case when $K=H$ is of particular interest. In this case
our necessary and sufficient condition takes the simple form stated in 
Corollary~\ref{cor:HxH} and Corollary~\ref{cor:action}.

We express our transition probabilities using a \emph{weight function} $w:G\rightarrow \mathbb{R}_{\ge 0}$. For a subset $S$ of $G$ we write $w(S)$ for $\sum_{s\in S}w(s)$. We stipulate that $w(G)>0$; hence we 
obtain a probability measure $Q$ on $G$ from $w$ by defining $Q(g) = w(g)/w(G)$ for all $g\in G$. We can therefore refer to the random walk on $G$ \emph{driven} by $w$. 
We work throughout with right multiplication
(the opposite convention to \cite{DRS}); thus the probability of a transition from $x \in G$ to $y \in G$
is $w(x^{-1}y)$.

\begin{theorem}\label{thm:HxK}
Let $w$ be a weight function on the finite group $G$, and let~$H$ and~$K$ be subgroups of $G$. A necessary and sufficient condition for the random walk on $G$ driven by $w$ to induce a time-homogeneous
Markov chain on the double cosets $H\backslash G/K$
is that $w(kxHyK) = w(xHyK)$ for all $x, y \in G$ and $k\in K$. 
If this condition holds then the probability of a step from $HxK$ to $HyK$ in the quotient chain is
proportional to $w(x^{-1}HyK)$.
\end{theorem}



When this condition holds we say that the random walk on $G$ \emph{lumps} on the double cosets $H\backslash G/K$.
In Definition~2.5 of \cite{Pang}, a chain with this property is said to have `strong lumping'.
It is possible for a quotient of a Markov chain to have the Markov property when the original chain
begins at a chosen distribution, but not lump in this strong sense: we refer
the reader to \cite[\S 6.4]{KemenySnell} and \cite{Pang} for further background on strong and weak lumping.

The principal interest of Theorem~\ref{thm:HxK} lies in the following corollary.

\begin{corollary}\label{cor:HxH}
Let $w$ be a weight function on the finite group
$G$ and let $H$ be a subgroup of $G$. A necessary and sufficient condition for the random walk on~$G$ driven by $w$ to
lump on the double cosets $H\backslash G/H$ is that $w(xH)=w(yH)$ whenever $HxH=HyH$. 
\end{corollary}

Often it is preferable to reason with group actions rather than cosets. In this context
Corollary~\ref{cor:HxH} implies the following result.

\begin{corollary}\label{cor:action}
Let $w$ be a weight function on the finite group
$G$. Suppose that $G$ acts transitively on a set $\Omega$. Let $\alpha \in \Omega$ have point stabiliser $H$.
\begin{thmlist}
\item The random walk on $G$ induces a time-homogeneous Markov chain on~$\Omega$ under the quotient map
$x \mapsto \alpha x$.
\item The induced Markov chain on $\Omega$ induces in turn a time-homogeneous
Markov chain on the $H$-orbits on $\Omega$
under the quotient map $\alpha \mapsto \alpha H$ if and only if
the probability of a step from $\beta$ to $\alpha$ is constant for $\beta$ in the same $H$-orbit.
\end{thmlist}
\end{corollary}

We prove these results in \S 2.
We remark that if $C$ is a conjugacy closed subset of the finite group $G$ then,
given two left cosets $xH$ and $hxH$ in the same double coset $HxH$, the map
$y \mapsto h y h^{-1}$ is a bijection between $C \cap xH$ and $C \cap hxH$. 
This proves the well known result that $C$ meets
each left coset $hxH$ in $HxH$ in the same number of elements. Thus
the special case 
$w(xH) = |C \cap xH|$ of Corollary~\ref{cor:HxH} is the motivating
result of Diaconis, Ram and Simper.
More generally, if $w$ is constant on the $K$-conjugacy classes in~$G$ then
\[ w(kxHyK) = w\bigl( (kxHyK)^k \bigr) = w(xHyKk) = w(xHyK) \]
as required in Theorem~\ref{thm:HxK};
this is the
sufficient condition for the random walk on $G$
to lump on the double cosets $H\backslash G/ K$
in Proposition 2.4 of \cite{DRS}.

As an application of Corollary~\ref{cor:action} 
we analyse a particular random walk on 
the symmetric group 
$\Sym_n$ related to the $r$-top to random shuffle studied in \cite{BWBell,
DFP, FulmanCardShuffling, Phatarfod}, in the special case $r=1$ in \cite[\S 3]{RobbinsBolker}
and \cite[Example 1]{BrownDiaconis}, and in a version permitting varying $r$ in \cite{BoardmanRudolfSaloffCoste}.
(These references are intended to be representative, particularly of recent work, but not exhaustive.)
 Let $r \in \N$ and let $n \ge 2r$.
In each step of the \emph{skewed $r$-random to top shuffle} on a deck
of $n$ cards, we make a uniform at
random selection of $r$ cards, none of them from the top~$r$ cards in the deck, 
and we move these~$r$ cards to the top (preserving their order).  
Using Corollary~\ref{cor:action}, we 
show that the shuffle induces a Markov chain on the double cosets $H\backslash G/H$
where $H$ is the subgroup $\Sym_r\times \Sym_{n-r}$ of $\Sym_n$. 
The set of possible shuffles is not a union of $H$-conjugacy classes:
therefore the full power of our necessary and sufficient condition for double coset lumping is required.
(We note however that this lumping on double cosets of $\Sym_r\times \Sym_{n-r}$ 
appears in \cite[\S 5]{DSim} and \cite[\S 8.2]{DRS} and Simper \cite{Simper} has studied it 
in the case where the driving set is all transpositions in $\Sym_n$.)
In \S 3 we find the transition matrix for the lumped
Markov chain, show that it is an involutive random walk in the sense of \cite{BWInv}, and thereby prove that it is irreducible, recurrent,  reversible and ergodic; 
we find its invariant distribution and its eigenvalues and hence its rate of convergence.
\emph{En route} we show that its transition matrix has the anti-diagonal eigenvalue property
studied in \cite{BWInv} and \cite{OchiaiEtAl}: this accounts for the remarkable spectral behaviour seen in~$(\star\star)$
before Example~\ref{ex:ex} and in Example~\ref{ex:ctd}.

\section{Proof of Theorems~\ref{thm:HxK}, Corollary~\ref{cor:HxH} and Corollary~\ref{cor:action}}

We shall require the following general fact about quotients of Markov chains,
proved as Theorem 6.4.1 in \cite{KemenySnell}. 

\begin{lemma}\label{lemma:Markov}
Let $M$ be a Markov chain on a set $\Omega$, and let $\mathcal{P}$ be a partition of $\Omega$ into subsets $\mathcal{P}_1,\dots, \mathcal{P}_t$. For $x\in \Omega$ and $j\in\{1,\dots,t\}$, let $p(x,\mathcal{P}_j)$ be 
the probability that $x$ transitions under $M$ to an element of $\mathcal{P}_j$. A necessary and 
sufficient condition for $M$ to induce a 
time-homogeneous Markov chain on the parts of $\mathcal{P}$, started
at an arbitrary initial distribution, is that $p(x,\mathcal{P}_j)=p(y,\mathcal{P}_j)$ 
whenever $x$ and $y$ lie in the same part of $\mathcal{P}$. 
\end{lemma}

\subsection{Proof of Theorem~\ref{thm:HxK}}
 Let $G$ be a finite group, and let $w:G\rightarrow \mathbb{R}_{\ge 0}$ be a weight function,
 normalized so that it is a probability measure. 
Let $M$ be the random walk on $G$ driven by $w$. Let~$H$ and $K$ be subgroups of $G$. 

Let $x$, $y\in G$. The probability $p(x,HyK)$ that $x$ transitions to an element of $HyK$ under $M$ is $w(x^{-1}HyK)$. By Lemma~\ref{lemma:Markov}, the random walk~$M$ induces a Markov chain on the double cosets $H\backslash G/K$ if and only if $p(x',HyK) = p(x,HyK)$ whenever $x'\in HxK$.   

Suppose that for all $x$, $y\in G$ and $k\in K$ we have $w(kxHyK)=w(xHyK)$. Suppose $g$, $g'\in G$ are such that 
$g'\in HgK$. Then we can write $g'=hgk^{-1}$ for some $h\in H$ and $k\in K$. For any $y\in G$ we have
\[
p(g',HyK) = w({g'}^{-1}HyK) = w(kg^{-1}h^{-1}HyK) = w(kg^{-1}HyK)
\]
which by supposition is equal to $w(g^{-1}HyK)$. So $p(g',HyK)=p(g,HyK)$, and it follows that  $M$ induces a Markov chain on $H\backslash G/K$. 

For the converse, suppose that there exist $x, y\in G$ and $k\in K$ such that $w(kxHyK)\ne w(xHyK)$. Then $p(x^{-1},HyK)\neq p(x^{-1}k^{-1}, HyK)$. But clearly $x^{-1}k^{-1}\in Hx^{-1}K$, and so the condition from Lemma~\ref{lemma:Markov} is not satisfied; hence $M$ does not induce a Markov chain on $H\backslash G/K$ in this case. This completes the proof of Theorem~\ref{thm:HxK}. 
\medskip

\subsection{Proof of Corollary~\ref{cor:HxH}}
Suppose that situation of the previous subsection obtains, but with the additional assumption that $K=H$. Suppose that~$M$ induces a Markov chain on $H\backslash G/H$. Then $w(kxHyH)=w(xHyH)$ for all $x,y\in G$ and $k\in H$, by Theorem~\ref{thm:HxK}. In the case that $y=1$, we have $w(kxH)=w(xH)$ for all $x\in G$ and $k\in H$. But for $y\in G$,  there exists $k\in H$ such that $yH=kxH$ if and only if $HxH = HyH$. This proves the necessity of the condition stated in the corollary.

For sufficiency of that condition, suppose that $w(xH)=w(yH)$ whenever $HxH=HyH$. Then $w(kxH)=w(xH)$ for all $x\in G$ and $k\in H$. Let $y\in G$. Letting $\{y_1,\dots, y_t\}$ be a transversal for the left cosets of $H$ in $HyH$, so
\[
HyH = \dot\bigcup_{j=1}^t y_jH,
\]
we see that
\[
w(kxHyH) = \sum_{j=1}^t w(kxy_jH) = \sum_{j=1}^t w(xy_jH) = w(xHyH).
\]
It now follows from Theorem~\ref{thm:HxK} that $M$ induces a Markov chain on $H\backslash G/H$, and the proof is complete.

\subsection{Proof of Corollary~\ref{cor:action}}
For (i), observe that if $K = \{1\}$ then the necessary and sufficient condition in Theorem~\ref{thm:HxK}
becomes $w(kxHy) = w(xHy)$ for all $k \in K$; this obviously holds. Therefore the random walk on $G$
induces a Markov chain on the right cosets $H \backslash G$, and since $H$ is the point-stabiliser of $\alpha$ in the
right action of $G$ on $\Omega$, the state space is isomorphic to $\Omega$ by the map $Hx \mapsto
\alpha x$.

For (ii), let $p(\omega, \alpha)$ be the probability of a step from $\omega$ to $\alpha$.
The coset~$xH$ is precisely those elements of $G$ that send $\alpha x^{-1}$ to $\alpha$.
Suppose that $w(xH) = w(yH)$ for all $x, y$ such that $HxH = HyH$.
Then, by setting $\beta = \alpha x^{-1}$ and $\gamma = \alpha y^{-1}$, we
see that $p(\beta,\alpha) = p(\gamma,\alpha)$ for all $\beta, \gamma$ in the same $H$-orbit.
The converse follows by reversing this argument.

\section{The skewed $r$-random to top shuffle}

Before introducing this shuffle, we review some results from \cite{BWInv} 
on weighted involutory walks which, remarkably enough, turn out to be closely connected.

\subsection{Involutory  walks}\label{subsec:involutive}
These walks, which can be defined on any poset~$P$ admitting an order-reversing involution~$\star$, have steps defined 
by two parts: starting at $x \in P$, a random point $y \le x$ is chosen
with probability prescribed by a weight function $\gamma$ on the intervals $[y,x] = \{u \in P 
: y \le u \le x\}$ of $P$; 
the step is completed by moving deterministically to $y^\star$.
The $\gamma$-weighted involutory walk on $P$ therefore transitions from $x$ to $y^\star$ only in the case that $y\le x$, and this transition has probability
\[
p(x,y^\star)  = \frac{\gamma_{[y,x]}}{N(\gamma)_x},
\]
where $N(\gamma)_x = \sum_{z\le x}\gamma_{[z,x]}$. 
We require that for each $x \in P$ there exists $y \le x$
such that $\gamma_{[y,x]}>0$. Thus $N(\gamma)_x > 0$ and the process is well-defined.
An important observation is that if~$F$ is a positive-valued function 
on~$P$ and~$\gamma'$ is the weight function
defined by $\gamma'_{[y,x]} =F(x)\gamma_{[y,x]}$ then 
$\gamma$ and~$\gamma'$
define the same weighted involutive walk, since
$N(\gamma')_x = F(x)N(\gamma)_x$ and so~$F(x)$ is cancelled under normalization. 

For $r\in \N$,
we consider the poset $\{0,1,\ldots,r\}$ with the natural order and anti-involution $x^\star = r-x$.
We require a two parameter family of 
weights on this poset, indexed by $a$, $b \in \N$ where $a \ge r$, and defined by
\[
\deltaP^{(a,b)}_{[y,x]} = {a \choose y}{b\choose x-y}.
\]
Observe that, by the condition on $a$, we have $N(\deltaP^{(a,b)})_x > 0$ for all $x \in \{0,1,\ldots, r\}$. 
For example, \smash{$\deltaP^{(a,1)}_{[x,x]} = \binom{a}{x}$}, \smash{$\deltaP^{(a,1)}_{[x-1,x]} = \binom{a}{x-1}$}
and \smash{$\deltaP^{(a,1)}_{[y,x]} = 0$} for $y \le x-2$.
In the $\deltaP^{(a,1)}$-weighted involutive walk on $\{0,1,\ldots, r\}$ we therefore have
$p(x,x^\star) = \binom{a}{x}/ \binom{a+1}{x} = (a-x+1)/(a+1)$, $p(x, (x-1)^\star) = 
\binom{a}{x-1}/\binom{a+1}{x}  = x/(a+1)$, and no other steps from $x$ are possible.

In the following lemma, we briefly summarise some of the results implied by  \cite{BWInv} on the 
$\deltaP^{(a,b)}$-weighted involutory walk on $\{0,1,\ldots,r\}$. We say that an $r\times r$ matrix~$A$ 
with rows and columns indexed by $\{0,1,\ldots, r\}$ with $(0,0)$ in the bottom right
has the \emph{anti-diagonal eigenvalue property} if $A_{ij}=0$ whenever $i+j > r$ (that is, entries strictly
to the left of the anti-diagonal are zero), and  the eigenvalues of $A$ are $(-1)^jA_{j(r-j)}$ for 
$j \in \{0,1,\ldots,r\}$.
See Example~\ref{ex:ex} for an example of such a matrix.

\begin{lemma}\label{lemma:delta}
Let $a, b \in \N$ be such that $a \ge r$. The $\deltaP^{(a,b)}$-weighted involutory walk on $\{0,1,\ldots,r\}$  is irreducible, recurrent, reversible and ergodic. It has unique invariant distribution
proportional to
\[
\pi_x={a\choose r-x}{a+b\choose x}.
\]
Its transition matrix has the anti-diagonal eigenvalue property, and its
 eigenvalues are $(-1)^j{a\choose j}/{a+b\choose j}$ for $j \in \{0,1,\ldots,r\}$. 
\end{lemma}

\begin{proof}
In the notation of \cite[(1.4)]{BWInv} 
$\deltaP^{(a,b)}$ is the weight denoted $\delta^{(a+1,b+1)}$.
Every claim in the lemma therefore follows from Theorem 1.5 in \cite{BWInv}.
\end{proof}

The walks related to the skewed $r$-random to top shuffle
that we consider here differ from the involutory walks in~\cite{BWInv} in that each step involves an $\emph{up-step}$ (rather than a down-step) before applying the involution. This is merely a consequence of the natural direction of the Markov chain; it is clear that it makes no important difference to the theory, and only slight difference to the notation. We shall write $\gamma_{[x,y]}$ rather than $\gamma_{[y,x]}$ for weights (preserving the convention that $x$ transitions to~$y$). The appropriate weight corresponding to $\deltaP^{(a,b)}$ in this context, obtained by applying the $\star$-involution,~is
\[
\overline\deltaP^{(a,b)}_{[x,y]} = {a \choose r-y}{b\choose y-x},
\]  
which has all the properties ascribed to $\deltaP^{(a,b)}$ in Lemma~\ref{lemma:delta}.

\subsection{Analysis of the skewed $r$-random to top shuffle}

Let $n$, $r\in \N$ with $n > 2r$. We consider shuffles of a deck of $n$ cards, 
with positions ordered from $1$ at the top to $n$ at the bottom.
A single \emph{skewed $r$-random to top shuffle} 
is as follows: $r$ cards lying strictly beneath position $r$ are chosen uniformly at random, and brought to the top of the deck, maintaining their order. Let~$D$ be the set of skewed $r$-random to top shuffles.

We are grateful to an anonymous referee for part of the following remark.

\begin{remark}
Consider the following
shuffle on an $n$ card deck. First cut the deck into three piles $A$, $B$, $C$, where $A$ is the top $r$ cards,
$B$ is the next $r$ cards, and $C$ is the remaining bottom $n-2r$ cards.
Riffle shuffle $A$ and $C$ in the usual way, with probabilities given by 
Gilbert--Shannon--Reeds model (see for instance \cite{BayerDiaconis}),
and then drop pile $B$ on top. Since in the inverse of a skewed $r$-random to top shuffle, the top $r$ cards
are moved to random positions in $\{r+1,\ldots, 2r\}$, maintaining their order,
this procedure performs an inverse skewed $r$-random to top shuffle.
\end{remark}

Repeated shuffles by~$D$ define the random walk on $\Sym_n$
driven by the weight $w$ where
\[
w(g) = \begin{cases} 1 & \text{if $g \in D$} \\ 0 & \text{otherwise.} \end{cases}
\]
Let $H = \Sym_r \times \Sym_{n-r}$ be the stabiliser of the ordered set partition
$\bigl( \{1,\dots, r\}, \{r+1,\dots,n\} \bigr)$. Observe that $H \backslash \Sym_r$
is in  bijection with the set of $r$-subsets of $\{1,\ldots, n\}$ by 
the map $Hg \mapsto \{1,\ldots, r\} g$. 
The double cosets in $H \backslash \Sym_n / H$ are indexed by $\{0,1,\ldots, r\}$: 
the double coset $HgH$ indexed by~$x$ contains exactly
those $g \in \Sym_n$ such that $\{1,\dots, r\}g\,\cap\, \{1,\dots, r\}$ has size~$x$. 
We say that an element of the double coset indexed by $x \in \{0,1,\ldots, r\}$ has \emph{type}~$x$.
For each $r$-subset $X$ of $\{1,\ldots, n\}$
not meeting $\{1,\ldots, r\}$ there exists a unique
$g \in D$ such that $Xg = \{1,\ldots, r\}$, while if $X$ meets $\{1,\ldots, r\}$
then $Xg \not= \{1,\ldots, r\}$ for any $g \in D$.
Therefore, by Corollary~\ref{cor:action}(i) and (ii), the skewed $r$-random to top shuffle
induces a Markov chain $M$ on the double cosets of $H$ in~$\Sym_n$.

We shall establish the following result.

\begin{theorem}\label{thm:involutive}
The Markov chain $M$ is a $\overline\deltaP^{(r,n-2r)}$-weighted involutory walk on 
the set of types $\{0,1,\ldots,r\}$.
\end{theorem}

\begin{proof}
Let $x$ and $z$ be types in $\{0,1,\ldots, r\}$; we consider the probability of transition from $x$ to $z$ in the Markov chain $M$. Let $g \in \Sym_n$ be a permutation of type~$x$
and let $d \in D$ be such that $gd$ has type $z$.
 There are $r-x$ points in $\{1,\dots, r\}$ whose image under $g$ lies in $\{r+1,\dots,n\}$. 
 If $X_g$ is the set $\{1,\dots,r\}g\cap\{r+1,\dots,n\}$, and if $R_d$ is the set of $r$ cards moved to the top by $d$, then $R_d$ must intersect in exactly $z$ points with $X_g$. The number of elements $d\in D$ satisfying this condition is 
\[
{r-x \choose z}{n-2r+x \choose r-z}, \tag{$\star$}
\]
since $|X_g|=r-x$ and the elements of $R_d$ 
not in $X_g$ are drawn from $\{r+1,\dots,n\}\setminus X_g$, which has size $(n-r) - (r-x)$.
The probability of transitioning from $x$ to $z$ is obtained by dividing the quantity
in ($\star$) by $|D|={n-r \choose r}$. 

Note that the transition probability is $0$ unless $z\le r-x$, or equivalently, $r-z \ge x$. Note also that the map $j\mapsto r-j$ is an order-reversing involution on $\{0,1,\ldots,r\}$. We can consider 
a single step of the walk on $\{0,1,\ldots,r\}$ as a two-part process: first transition from $x$ to an element $y$ in the up-set of~$x$; then transition from $y$ to $z=r-y$. The walk is therefore a $\gamma$-weighted involutory walk 
(defined using up-steps as explained in \S\ref{subsec:involutive}).

A weight function $\gamma$ on intervals for this walk
is obtained by substituting~$y$ for $r-z$ in ($\star$): by the observation on scaling earlier, we can take 
\[
\gamma_{[x,y]} = {r-x \choose r-y}{n-2r+x \choose y}F(x),
\]
where $F(x)$ is any positive-valued function of $x$. In particular, if we take
\[
F(x) = \frac{r!(n-2r)!}{(r-x)!(n-2r+x)!},
\]
then we obtain the weight
\[
\gamma_{[x,y]} = \frac{r!(n-2r)!}{(r-y)!(y-x)!y!\bigl( n-2r+x-y\bigr)!} = {r \choose y}{n-2r \choose y - x}.
\]
Hence $\gamma=\overline\deltaP^{(r,n-2r)}$, and this completes the proof.
\end{proof}

It now follows from Lemma~\ref{lemma:delta} that the the Markov chain
$M$ is reversible, irreducible, recurrent and ergodic.
It is notable that $M$ is reversible: note that the random walk
on $\Sym_n$ from which $M$ is induced is \emph{not} reversible because $D$ is not closed
under taking inverses. Likewise the walk on $r$-subsets of $\{1,\ldots, r\}$ it induces 
again has constant invariant distribution and so is not reversible, by consideration of the detailed balanced
equations. 
By Lemma~\ref{lemma:delta},
the transition matrix of $M$ has the anti-diagonal eigenvalue property, and  its eigenvalues 
are 
\[ 
\frac{ (-1)^{r-x}{r\choose x}}{{n-r\choose r-x}} \tag{$\star\star$}
\]
for $x\in \{0,1,\ldots,r\}$. Its unique invariant distribution
$\pi_x$ is proportional to ${a\choose r-x}{a+b\choose x}$.
The scaling factor is, by Vandermonde's convolution, $\binom{2a+b}{r}$.

In particular the second largest eigenvalue in absolute value of the transition matrix
is~$-r/(n-r)$. This controls the rate of convergence
of $M$. More precisely, by Corollary~12.7 in~\cite{LevinPeres},
\[ \lim_{t \rightarrow \infty}\, \Bigl|\Bigl| \frac{\P[M_t = x]}{\pi_x} - 1 \Bigr|\Bigr|^{1/t}
 = \frac{r}{n-r} \]
for any $x \in \{0,1,\ldots, r\}$. Using a general
result on reversible Markov chains which, to make the article self-contained, we now set up,
we can prove something much stronger. (We refer the reader to 
\cite[\S 4]{Rosenthal} for background and to \cite{DSal} for 
extensions to this argument.)
Let $\langle - , -\rangle$ be the inner product on $\mathbb{R}^{r+1}$, indexing positions from $0$, defined by
\[ \langle \theta, \phi \rangle = \sum_{z=0}^r \frac{\theta_z \phi_z}{\pi_z}. \]
Since $M$ is reversible, its transition matrix $Q$ satisfies
the detailed balance equations $\pi_x Q_{xy} = \pi_y Q_{yx}$ for all $x, y \in \{0,1,\ldots, r\}$.
Therefore 
\[ \langle \theta Q, \phi \rangle = \sum_{y=0}^r \sum_{x=0}^r \theta_x \phi_y \frac{Q_{xy}}{\pi_y}
= \sum_{x=0}^r \sum_{y=0}^r  \theta_x \phi_y \frac{Q_{yx}}{\pi_x} = \langle \theta, \phi Q \rangle \]
and so right multiplication by $Q$ is self-adjoint with respect to $\langle -, -\rangle$.
In particular the right eigenvectors of $Q$ are orthogonal. Since $\pi$ spans the $1$-eigenspace
we have $\sum_{x=0}^r \phi_x = \langle \pi, \phi \rangle = 0$ for any right-eigenvector
$\phi$ not proportional to $\pi$. Hence, given an arbitrary starting distribution 
$\theta$, there exist unique scalars $c_\lambda \in \mathbb{R}$ such that
$\theta = \pi + \sum_{\lambda \not = 1} c_\lambda \phi^{(\lambda)}$
where the sum is over the non-unit eigenvalues from~($\star\star$)
and \smash{$\phi^{(\lambda)}$} is a right-eigenvector of $Q$ of eigenvalue $\lambda$ of norm $1$
in $\langle -, -\rangle$. Applying $Q^t$ we obtain
$\bigl\langle \theta Q^t - \pi, \theta Q^t - \pi \bigr\rangle =
\sum_{\lambda \not=1} c_\lambda^2 \lambda^{2t}$. 
Thus using the definition of $\langle -, -\rangle$ 
and $(\theta Q^t)_x = \P[M_t = x | M_0 \sim \theta]$, we obtain the sharp equality
\[ \sum_{x=0}^r  \frac{( \P[M_t = x | M_0 \sim \theta]- \pi_x)^2}{\pi_x} =
\sum_{\lambda\not=1} c_\lambda^2 \lambda^{2t}. \tag{$\ddagger$}\]
Therefore, for each $x$, we have
$(\P[M_t = x | M_0 \sim \theta] - \pi_x)^2 \le 
\pi_x \sum_{\lambda \not=1} c_\lambda^2 \lambda^{2t}. $
Taking the square-root and summing over $x$ using $\sum_{x=0}^r \sqrt{\pi_x} \le \sqrt{r+1}$, we get
\[ \Bigl|\Bigl| \P[M_t = x | M_0 \sim \theta]_{0 \le x \le r} - \pi \Bigr|\Bigr|_{\mathrm{TV}} \le 
\frac{\sqrt{r+1}}{2} \sqrt{\,\sum_{\lambda \not=1} c_\lambda^2 \lambda^{2t}}. \]
This gives an optimal bound, up to the constant $\mfrac{1}{2}\sqrt{r+1}$, on the convergence 
to the invariant distribution in the total variation distance.



We end with an example of the transition matrix of this Markov chain,
and a natural generalization.

\begin{example}\label{ex:ex}
Let $n = 10$ and $r = 4$. The transition matrix of $M$ with rows and columns labelled $\{0,1,2,3,4\}$
with $(0,0)$ at the bottom-right is
\[ \left( \begin{matrix} 
	\cdot & \cdot & \cdot & \cdot & 1 \\[2pt]
	\cdot & \cdot & \cdot & \frac{2}{3} & \frac{1}{3} \\[2pt]
	\cdot & \cdot & \frac{2}{5} & \frac{8}{15} & \frac{1}{15} \\[2pt]
	\cdot & \frac{1}{5} & \frac{3}{5} & \frac{1}{5} & 0 \\[2pt]
	\frac{1}{15} & \frac{8}{15} & \frac{2}{5} & 0 & 0
 \end{matrix} \right) \]
where $\cdot$ denotes a zero entry implied by the observation in the proof of Theorem~\ref{thm:involutive} 
that if $x$ steps to $z$ then $z \le r-x$.
By the anti-diagonal eigenvalue property, the eigenvalues are $1$, $-\mfrac{2}{3}$, $\mfrac{2}{5}$,
$-\mfrac{1}{5}$, $\mfrac{1}{15}$. The unique invariant distribution to which the Markov chain
converges is $\mfrac{1}{210}(1, 24, 90, 80, 15)$. To illustrate $(\ddagger)$, 
we take the starting distribution $\theta$ to be $(\mfrac{1}{35}, \mfrac{4}{35}, \mfrac{2}{7}, \mfrac{4}{7}, 0)$,
chosen so that 
\smash{$\theta = \pi + \frac{1}{\sqrt{3}} v^{(-1/5)}$}.
The right-hand
side of $(\ddagger)$ is $5^{-2t}/3$ and 
\smash{$\bigl|\bigl| \P[M_t = x | M_0 \sim \theta]_{0 \le x \le r} - \pi 
\bigr|\bigr|_{\mathrm{TV}} \le \mfrac{\sqrt{5}}{2\sqrt{3}} 5^{-t}$}.
\end{example}

More generally, we may consider, for each $\ts \in \{0,1,\ldots, r\}$, the \emph{type $\ts$ skewed $r$-random
to top} shuffle, in which~$r-\ts$ cards are chosen from the top~$r$ cards in the deck,~$\ts$ cards from
the bottom $n-r$ cards in the deck, and all $r$ cards are brought to the top, preserving their
relative order. 
One can show that all these shuffles also lump on double cosets, and, because the Hecke algebra of double
cosets of $\Sym_r \times \Sym_{n-r}$ in $\Sym_n$ is commutative, the transition matrices of the lumped
random walks commute. 

\begin{example}\label{ex:ctd}
Example~\ref{ex:ex} shows the transition matrix of the $4$-skewed $4$-random to top shuffle.
The transition matrices of the type $1$-, $2$- and $3$-skewed $4$-random
to top shuffles when $n=10$ are shown below, again
with $(0,0)$ at the bottom right. (The type~$0$ skewed $4$-random to top shuffle is the identity.)
\begin{align*} &\left( \begin{matrix} 
	\zero & 1 & \zero & \zero & \zero \\
	\frac{1}{24} & \frac{1}{3} & \frac{5}{8} & \zero & \zero \\[2pt]
	\zero & \frac{1}{6} & \frac{1}{2} & \frac{1}{3} & \zero \\[2pt]
	\zero & \zero & \frac{3}{8} & \frac{1}{2} & \frac{1}{8} \\[2pt]
	\zero & \zero & \zero & \frac{2}{3} & \frac{1}{3} \end{matrix} \right) \quad
   \left( \begin{matrix}
    \zero & \zero & 1 & \zero & \zero \\
    \zero & \frac{1}{6} & \frac{1}{2} & \frac{1}{3} & \zero \\[2pt]
    \frac{1}{90} & \frac{2}{15} & \frac{13}{30} & \frac{16}{45} & \frac{1}{15} \\[2pt]
    \zero & \frac{1}{10} & \frac{2}{5} & \frac{2}{5} & \frac{1}{10} \\[2pt]
    \zero & \zero & \frac{2}{5} & \frac{8}{15} & \frac{1}{15} \end{matrix} \right) \quad
   \left( \begin{matrix}
    \zero & \zero & \zero &  1 & \zero \\
    \zero & \zero & \frac{3}{8} & \frac{1}{2} & \frac{1}{8} \\[2pt] 
    \zero & \frac{1}{10} & \frac{2}{5} & \frac{2}{5} & \frac{1}{10} \\[2pt]
    \frac{1}{80} & \frac{3}{20} & \frac{9}{20} & \frac{7}{20} & \frac{3}{80} \\[2pt]
    \zero & \frac{1}{5} & \frac{3}{5} & \frac{1}{5} & \zero \end{matrix}\right) \\
&   \hspace*{15pt} (1, \mfrac{7}{12}, \mfrac{1}{4}, 0, -\mfrac{1}{6}) \hspace*{0.50in}
           (1, \mfrac{1}{6},  -\mfrac{1}{10}, -\mfrac{1}{15}, \mfrac{1}{15}) \hspace*{0.44in}
           (1, -\mfrac{1}{4}, -\mfrac{1}{20}, \mfrac{1}{10}, -\mfrac{1}{20})
\end{align*}
Since these matrices commute, they admit a simultaneous eigenbasis. 
The eigenvalues,
in the order corresponding to the eigenvalues 
$(1, -\mfrac{2}{3}, \mfrac{2}{5}, -\mfrac{1}{5}, \mfrac{1}{15})$
of the skewed $4$-random to top shuffle seen in Example~\ref{ex:ex}, are shown below each matrix.
If we choose
a skewed shuffle of type $\ts$ with probability proportional to $\binom{r}{\ts}$ then we obtain
the usual $r$-random to top shuffle. Thus the weighted sum of the five skewed shuffle matrices
is the transition matrix of the $r$-random to top shuffle, lumped on double cosets.
In particular, its eigenvalues, namely
$(1, \mfrac{1}{6}, \mfrac{1}{10}, \mfrac{1}{20}, \mfrac{3}{80})$,
are the weighted sums of the eigenvalues for the five skewed shuffle matrices.
This should help to explain why the skewed $r$-random to top shuffle and its
generalizations have negative eigenvalues, even though typically shuffle matrices
have only non-negative eigenvalues.
\end{example}

\section*{Acknowledgements}
We thank an anonymous referee for very helpful comments and a careful reading of the paper.
We also thank Edward Crane for helpful comments.

\end{document}